\documentclass[12pt,reqno]{amsart}
\usepackage{amssymb,amsmath,amsthm}
\usepackage{graphicx}
\usepackage{subcaption}
\usepackage{fullpage}
\usepackage{scrextend}
\usepackage{siunitx}
\usepackage{enumerate}
\usepackage{tikz,calc}
\usepackage{tikz-cd}
\usepackage{subcaption}
\usetikzlibrary{calc}
\usepackage{epsfig,graphicx}
\usepackage{listings}
\usepackage{enumerate}

\usepackage[hypertexnames=false]{hyperref}

 \newtheorem{theorem}{Theorem}
\newtheorem{prop}[theorem]{Proposition}
\newtheorem{lemma}[theorem]{Lemma}

\newtheorem{alg}[theorem]{Algorithm}

\begin{document}

\title{Grid-drawings of graphs in three-dimensions}

\author{J\'ozsef Balogh}
\address{Department of Mathematics \\ University of Illinois Urbana-Champaign \\ 1409 W. Green Street, Urbana IL 61801 \\ United States}
\email{jobal@illinois.edu}
\thanks{Balogh was supported in part by NSF grants DMS-1764123 and RTG DMS-1937241, FRG DMS-2152488, the Arnold O. Beckman Research Award (UIUC Campus Research Board RB 24012).}
\author{Ethan Patrick White}
 \address{Department of Mathematics \\ University of Illinois Urbana-Champaign \\ 1409 W. Green Street, Urbana IL 61801 \\ United States}
\email{epw@illinois.edu}
\thanks{White is supported in part by an NSERC Postdoctoral Fellowship}

\begin{abstract}
Using probabilistic methods, we obtain grid-drawings of graphs without crossings with low volume and small aspect ratio. We show that every $D$-degenerate graph on $n$ vertices can be drawn in $[m]^3$ where $m^3 = O(D^2 n\log n)$. In particular, every graph of bounded maximum degree can be drawn in a grid with volume $O(n \log n)$. 


\end{abstract}

\maketitle

\section{Introduction}

A \emph{grid-drawing} of a graph $G$ is a representation of $G$ where vertices are distinct points of $\mathbb{Z}^d$ and edges are straight-line segments such that 

\begin{enumerate}[(i)]
\item no edge intersects a vertex that it is not adjacent to in $G$;
\item no pair of edges have a common interior point.
\end{enumerate}


\noindent The study of three-dimensional grid-drawings is in part motivated by problems in circuit architecture~\cite{LR} and data visualization~\cite{WM}.\\

 For a positive integer $m$, let $[m]$ denote the set $\{0,1,\ldots,m-1\}$. In contrast to the planar setting, a construction using a moment curve shows that every $n$-vertex graph can be drawn in a $[n] \times [2n] \times [2n]$ grid, and the size of this grid is tight up to a multiplicative constant~\cite{CELR}. Grid-drawings have been studied with respect to many graph parameters, including maximum degree, genus, path-width, and chromatic number~\cite{DJMMUW,DMW,DWood04,DWood06,PTT}. Typically, we are interested in grid-drawings of minimum volume, and also small aspect ratio, i.e., the ratio between the longest and shortest sides of the grid. The degeneracy of a graph $G$ is the smallest positive integer $D$ such that every subgraph of $G$ has a vertex of degree at most $D$. A long-standing open problem of Pach, Thiele, and T\'oth~\cite{PTT} is to determine if all graphs of maximum degree $3$ have a grid-drawing with $O(n)$ volume. Our main result resolves this problem up to logarithmic factors.

\begin{theorem}\label{blowupThm} Let $G$ be a $D$-degenerate graph with $n$-vertices and $k \geq n$ edges. Then there is a grid-drawing of $G$ in $[m]^3$ where $m^3 = O(D k \log n )$. 
\end{theorem}

Bose, Czyzowicz, Morin, and Wood~\cite{BCMW} show that the maximum number of edges in a grid-drawing is bounded by eight times the volume of the grid, and therefore our bound is tight up to the logarithmic and degeneracy factors. The previous best bound on the volume of a grid-drawing of bounded degeneracy graphs was $O(n^{3/2})$. More precisely, it is known that there exists a grid-drawing of every $D$-degenerate graph with $n$ vertices and $k$ edges in grids of volume $O(D^{15/2}nk^{1/2})$ and $O(Dnk)$~\cite{DWood04,DWood06}. The grid-drawings behind these two earlier bounds have large aspect ratios; there is a side of the grid with length $\Omega(n)$. \\



Felsner, Liotta, and Wismath~\cite{FLW} posed the problem of determining if all planar graphs on $n$ vertices can be drawn in a 3-dimensional grid with $O(n)$ volume. All progress on this problem was made by constructing so-called \emph{track-layouts}. Recent progress includes a volume bound of $O(n \log^8 n)$ of Di Battista, Frati, and Pach~\cite{DFP}, and a bound of $O(n \log n)$ of Dujmovi\'c~\cite{D15}. A breakthrough of Dujmovi\'c  Joret,  Micek,  Morin,  Ueckerdt, Wood~\cite{DJMMUW} fully resolved this conjecture by showing there exists a drawing of every $n$-vertex planar graph in a grid with dimensions $O(1) \times O(1) \times O(n)$. In general, the track-layout method produces a grid-drawing with one very large dimension. In~\cite{CELR} the authors pose the problem of deciding if every 3-dimensional grid-drawing of an $n$-vertex planar graph requires all sides of the grid to be $\Omega(n)$. Since planar graphs are $5$-degenerate, Theorem~\ref{blowupThm} shows that every planar graph can be drawn in a grid with all sides $O(n^{1/3}\log^{1/3} n)$, thereby resolving this problem in a strong way. \\


Our method is to use random embeddings of graphs in grids. To determine the probability that our embeddings do not have edge crossings we require estimates on the number of point-tuples on affine subspaces in grids. In Section~\ref{tupleSec} we prove the necessary lemmata about point-tuples. In Section~\ref{drawingSec} we verify our probabilistic constructions. In Section~\ref{concSec} we discuss continuations of this work and some other open problems. 

\section{Counting tuples on affine subspaces}\label{tupleSec}

Edges that pass through other vertices are forbidden in our drawings. As a result, we require an estimate on collinear triples. Versions of the following lemma have appeared in the literature, see for example~\cite{Z}.

\begin{lemma}\label{collinearCount}
Let $d,k,m \geq 2$ be positive integers and $C_{d,k,m}$ denote the number of collinear $k$-tuples in $[m]^d$. Then 
\[ C_{d,k,m} = \begin{cases} O(m^{d+k-1}) & \text{ if } k \geq d+ 2 \\ O(m^{d+k-1} \log m ) & \text{ if } k = d+1 \\ O(m^{2d}) & \text{ if } d \geq k, \end{cases}\]
where the implicit constants may depend on $d$ and $k$ but not $m$. 
\end{lemma}

\begin{proof} By the symmetry of $[m]^d$, it is sufficient to count collinear tuples on a line through the origin and then multiply by $m^d$. Let $v \in \mathbb{Z}^d$ be the direction of a line intersecting at least $k$ points of $[m]^d$. Note that $\| v\|_\infty \leq m$. For each $1 \leq s \leq m$, the number of choices for $v$ such that $ \| v \|_{\infty} = s$ is at most $ds^{d-1}$. Since the number of points of $[m]^d$ on a line in direction $v$ is $\lceil m/\|v\|_{\infty} \rceil$, the number of collinear tuples on lines through the origin is at most 
\[ \sum_{s = 1}^m ds^{d-1} \left\lceil \frac{m}{s} \right\rceil^{k-1}  \leq d(2m)^{k-1} \sum_{s=1}^m s^{d-k} .\]
After multiplying by $m^d$, the claimed bounds follow.

\end{proof}

If two edges cross in a three-dimensional graph drawing, then these edges lie in a common plane. For this reason we need an estimate on the number of coplanar four-tuples in grids. In the setting of arbitrary $N$ point sets with at most $O(N^{1/2})$ coplanar points, Cardinal, T\'oth, and Wood obtain a bound on the number of coplanar four-tuples in $\mathbb{R}^d$ that matches the bound given by our Proposition~\ref{3dim4tup} for $d = 3$~\cite[Lemma 4.5]{CTW}. A recent result of Suk and Zeng improves the upper bound for the number of points in general position in $d$-dimensional integer grids~\cite[Theorem 1.3]{SZ}.




\begin{lemma}\label{hyperplaneBound} 
Let $d \geq 2$ be a positive integer, and $a_1,\ldots,a_d \in \mathbb{Z}$ not be all zero with greatest common denominator 1. Let 
\[ \mathcal{L} = \{ (x_1,\ldots,x_d) \in \mathbb{Z}^d \colon \sum_{i=1}^d a_ix_i = 0 \}.\]
Let $s = \max_i\{|a_i|\}$, and $m \geq s$ be a positive integer. If $\mathcal{L} \cap [m]^d$ spans a $(d-1)$-dimensional subspace, then $|\mathcal{L} \cap [m]^d| \leq 3^dm^{d-1}/s$. 

\end{lemma}

\begin{proof} Since $\mathcal{L}$ is a $(d-1)$-dimensional subspace we can find linearly independent vectors $v_1,\ldots,v_{d-1} \in \mathbb{Z}^d$ such that 
\[ \mathcal{L}  = \text{span}_\mathbb{Z} \{ v_1,\ldots,v_{d-1}\}.  \] 
Moreover, since $\mathcal{L} \cap [m]^d$ spans a $(d-1)$-dimensional subspace the $v_i$ can be chosen such that $\|v_i\|_\infty \leq m$ for all $1 \leq i \leq d-1$. Let $b \in \mathbb{R}^d$ be the unique vector such that for all $x \in \mathbb{R}^d$:
\begin{equation}\label{bDef}
\langle b,x \rangle = \det (v_1,\ldots,v_{d-1},x). 
\end{equation}
It is straightforward to check that $b$ exists and is unique by taking $x$ to be standard basis vectors. Since $b$ is perpendicular to $\mathcal{L}$, we see that $a = (a_1,\ldots,a_d)$ and $b$ must be parallel. Moreover, since $b \in \mathbb{Z}^d$ and the greatest common denominator of the entries in $a$ is 1, we conclude that $|b_i| \geq |a_i|$ for all $1 \leq i \leq d$. Suppose that $a_d \neq 0$ and so the projection of $\mathcal{L}$ to the first $d-1$ coordinates, call it $\mathcal{L}_d$, is a bijection. For $1 \leq i \leq d-1$, let $v_i^{(d)} \in \mathbb{R}^{d-1}$ be the vector $v_i$ with $d^{th}$ coordinate deleted. Observe that 
\[ \mathcal{L}_d = \text{span}_\mathbb{Z} \{ v_1^{(d)},\ldots,v_{d-1}^{(d)}\}, \] 
and $|\mathcal{L} \cap [m]^d| = |\mathcal{L}_d \cap [m]^{d-1}|$. The fundamental parallelepiped of $\mathcal{L}_d$ is the set 
\[ \mathcal{P}_d = \left\{ \sum_{i=1}^{d-1} x_i v_i^{(d)}  \colon 0 \leq x_i < 1 \right\}.\]
Let 
\[ \mathcal{T}_d = \bigcup_{\substack{t \in \mathcal{L}_d \\ (\mathcal{P}_d + t) \cap \mathcal{L}_d \cap [m]^{d-1} \neq \emptyset }} (\mathcal{P}_d + t ) ,\]

\noindent be the union of all integer vector translates of $\mathcal{P}_d$ that have nonempty intersection with $\mathcal{L}_d \cap [m]^{d-1}$. Since all entries of all vectors in $\{v_i\}_{1 \leq i \leq d-1}$ at most $m$ in absolute value, we conclude that $\mathcal{T}_d \subseteq [-m,2m]^{d-1}$. Hence the number of translates of $\mathcal{P}_d$ comprising $\mathcal{T}_d$ is at most $(3m)^{d-1}/|\mathcal{P}_d|$, and consequently $(3m)^{d-1}/|\mathcal{P}_d|$ is also an upper bound for $|\mathcal{L}_d \cap [m]^{d-1}|$. Note that from~\eqref{bDef} by taking $x$ to be the $d^{th}$ elementary basis vector, we see $|b_d| = \det(v_1^{(d)}, \ldots, v_{d}^{(d)}) = |\mathcal{P}_d|$. From $|a_d| \leq |b_d|$ we conclude $|\mathcal{L} \cap [m]^{d}| \leq 3^d m^{d-1}/|a_d|$. Moreover, as this argument can repeated for any non-zero coordinate of $a$, we conclude that $|\mathcal{L} \cap [m]^{d}| \leq 3^d m^{d-1}/s$.  

\end{proof}


\begin{prop}\label{3dim4tup}
Let $m$ be a positive integer, the number of 3-tuples in $[m]^3$ that are coplanar with the origin is $O(m^6 \log m)$. Moreover, the number of coplanar 4-tuples in $[m]^3$ is $O(m^9\log m)$.

\end{prop}

\begin{proof} The claimed bound on the number of coplanar $4$-tuples follows from the bound on $3$-tuples on a common plane through the origin by symmetry of $[m]^3$. We may restrict our attention to the case when the $3$-tuple spans a $2$-dimensional subspace, since the number of collinear $3$-tuples through the origin is $O(m^3 \log m)$ by Lemma~\ref{collinearCount}. Let $a \in \mathbb{Z}^3$ be the normal of a plane through the origin intersecting at least three other points in $[m]^3$. By Lemma~\ref{hyperplaneBound} the number of points on such a plane is at most $27m^2/\|a \|_\infty$. Since the vector $a$ can be determined as the cross-product of two integer vectors in $[m]^3$, we know that $\|a\|_{\infty} \leq m^2$. For $1 \leq s \leq m^2$, the number of choices of such normal vector $a$ with $\|a\|_{\infty} = s$ is at most $3s^2$. Putting these estimates together, the number of $3$-tuples on a plane through the origin is bounded above by
\[ \sum_{s = 1}^{m^2} \left( \frac{27m^2}{s} \right)^3 \cdot 3s^2 = O(m^6\log m).  \]

\end{proof}

\section{Probabilistic grid-drawings}\label{drawingSec}

All of our constructions use randomness. For an $n$-vertex graph $G$ and grid $[m]^3$ we define a \emph{random embedding} to be a uniformly sampled embedding of the vertices of $G$ into $[m]^3$ out of all $\binom{m^3}{n}n!$ possible embeddings. Throughout this section we will use $\hat{G}$ to denote a random embedding of $G$. Given a random embedding $\hat{G}$ in $[m]^3$, we say an edge is in a \emph{conflict} if either (i) another edge intersects its interior; or (ii) it contains a vertex in its interior. If no edge has a conflict in an embedding, then $\hat{G}$ is a grid-drawing. Our next result is a toy-version of our main theorem. It uses the tools from Section~\ref{tupleSec} in a similar way to the proof of Theorem~\ref{blowupThm}.

\begin{theorem}\label{firstMoment3D} There exists an absolute constant $C$ such that the following holds. For every graph $G$ on $n$ vertices with $k$ edges, if $m$ is a positive integer such that $m \geq C((nk)^{1/3} + k^{2/3} (\log k)^{1/3})$, then there is a grid-drawing of $G_n$ in $[m]^3$.

\end{theorem} 

\begin{proof} Let $G$ be an $n$-vertex graph with $k$ edges, $m > n^{1/3}$ be a positive integer to be determined later, and $\hat{G}$ be a random grid-drawing of $G$ in $[m]^3$. We estimate the expected number of edges of $\hat{G}$ in a conflict. Let $e,
u$ be an edge and vertex in $G$ where $u$ is not incident to $e$. Fix a triple of collinear points $x,y,z \in [m]^3$ where $y$ is between $x,z$. The probability that in $\hat{G}$ the endpoints of $e$ map to $x,z$ and $u$ maps to $y$ is 
\[ \frac{2(m^3-3)(m^3-4)\cdots (m^3 - n+1)}{m^3(m^3-1)\cdots(m^3-n+1)} = O(m^{-9}). \]
By Lemma~\ref{collinearCount}, the number of collinear triples in $[m]^3$ is $O(m^6)$. Since the number of edge-vertex pairs in $G$ is $nk$ we conclude that the expected number of edge-vertex pairs in $\hat{G}$ such that the vertex is in the interior of the edge is $O(nkm^{-3})$. Similarly, the probability that a particular pair of edges in $G$ are randomly assigned to a particular collinear $4$-tuple in $[m]^3$ is 
\[ \frac{24(m^3-4)(m^3-5)\cdots(m^3-n+1)}{m^3(m^3-1)\cdots(m^3-n+1)} = O(m^{-12}).\]
By Proposition~\ref{3dim4tup}, there are $O(m^9\log m)$ collinear 4-tuples in $[m]^3$. Since the number of edge pairs is $O(k^2)$ we see that the expected number of edges in conflict is 
\[ O\left( m^{-3} \cdot nk + m^{-3} \log m \cdot k^2 \right).\]
It follows that there exists a constant $C$ such that if $m \geq C((nk)^{1/3} + k^{2/3} (\log k)^{1/3})$, then the expected number of edges in a conflict is less than $1$, and so there is a grid-drawing of $G$ in $[m]^3$. 

\end{proof}

Since the number of edges in a planar graph is at most $3n$, Theorem~\ref{firstMoment3D} shows that every planar graph can be drawn in a grid with all sides $O(n^{2/3} (\log n)^{2/3})$, thereby answering Question 2 of Cohen, Eades, Lin, and Ruskey~\cite{CELR}. To prove our main theorem, we begin with a random embedding of the blow-up of a graph, and then greedily find a copy of our graph without any edge conflicts. \\

Let $G$ be an $n$-vertex graph with vertices $v_1, \ldots, v_n$. For a positive integer $t$, the $t$-blowup $G(t)$ of $G$ has $nt$ vertices, where the vertices are partitioned into $n$ sets $V_1,\ldots, V_n$ each of size $t$. If $v_iv_j$ is an edge in $G$ then every pair of vertices in $V_i,V_j$ is an edge. For a subset of vertices $S$ in $G$ we denote by $G[S]$  the subgraph of $G$ induced by $S$. If $G$ is a $D$-degenerate graph then there exists an ordering of the vertices $v_1,\ldots,v_n$ of $G$ such that the degree of $v_i$ in $G[v_i,\ldots,v_n]$ is at most $D$. For the rest of this section, the vertices of $D$-degenerate graphs will always be labeled in this way. This ordering is called a \emph{degeneracy ordering}. \\


By Proposition~\ref{3dim4tup} there is a constant $C_1$ such that for any point $v \in [m]^3$ the number of triples coplanar to $v$ is at most $C_1 m^6 \log m$. By Lemma~\ref{collinearCount}, there is a constant $C_2$ such that for any point $v \in [m]^3$ the number of pairs of points collinear to $v$ is at most $C_2 m^3$. We will use the constants $C_1,C_2$ below. The main idea behind the proof of Theorem~\ref{blowupThm} can be described by the following randomized algorithm.

\setcounter{theorem}{0}
\begin{alg}\rm\label{boundedDegAlg}~\\
\vspace{-0.5cm}
\begin{enumerate}[1:]
    \item \textbf{Input:} $n$-vertex graph $G$ with degeneracy $D$ and $k \geq n$ edges.
    \item Put $t =\max\{ \lceil  \log n\rceil , \lceil D \log D \rceil\}$. Find minimum $m \in \mathbb{N}$ such that 
    \[ m^3 \geq 1000 C_1D k \log n .\]


    
    \item Construct $\hat{G}(t)$, a random embedding of the $t$-blowup $G(t)$ in $[m]^3$. Let $V_1,\ldots,V_n$ be the sets of $t$ vertices corresponding to a vertex in $G$, where these sets are labelled according to a degeneracy ordering for $G$.

    \item Initialize $H = \emptyset$. For $i = n,n-1 \ldots, 1$: 
    \begin{itemize}
        \item Select an embedded vertex in $\hat{v}_i \in V_i$ and add it to $H$. Vertex $\hat{v}_i$ must be such that none of the edges in $\hat{G}(t)[H]$ is in a conflict in $\hat{G}(t)[H]$. 
    \end{itemize}
    \item \textbf{Output:} $H$. 
\end{enumerate}
\end{alg}

Note that if Algorithm~\ref{boundedDegAlg} successfully terminates, then the output $H$ will be a grid-drawing of $G$ in $[m]^3$ meeting the criterion of Theorem~\ref{blowupThm}.


\begin{proof}[Proof of Theorem~\ref{blowupThm}]



Let $m, t$ be as in Step 2 of Algorithm~\ref{boundedDegAlg}. Note that since $k \geq n$, we have $m^3> nt$ and so it is possible to embed $G(t)$ in $[m]^3$. Let $\hat{G}(t)$ be a random embedding of $G(t)$ in $[m]^3$. Consider the event that Step 4 cannot be completed for an iteration $1 \leq i \leq n$. Let $\{v_j\}_{j \in S}$ be the neighbors of $v_i$ in $G[v_{i+1},\ldots,v_n]$, and let $\{\hat{v}_j\}_{j \in S}$ be the previously selected embedded vertices in parts $V_j$ for every $j \in S$. Note that $|S| \leq D$. Let $T_j \subseteq V_i$ be the set of vertices such that if $\hat{w} \in T_j$, the edge $\hat{w}\hat{v}_j$ is in a conflict in $\hat{G}(t)$. The fact that iteration $i$ of Step 4 cannot be completed is equivalent to $\cup_{j \in S} T_j = V_i$. Repeatedly remove vertices from a set $T_j$ if they appear in a $T_{j'}$ with $j'<j$ so that the previous union is a partition. For every $j \in S$ let $T_j^1 \subseteq T_j$ be the set of vertices in $T_j$ that form an edge-crossing conflict with $\hat{v}_j$ and $T_j^2 \subseteq T_j\setminus T_j^1$ be the set of vertices in $T_j$ that form a vertex-interior conflict with $\hat{v}_j$.  \\


\emph{Case(i):} $N_1 = \sum_{j \in S} |T_j^1| \geq t/2$. For every $j \in S$, fix arbitrary tuples of edges in $G(t)$, indexed by the vertices of $T_j^1$: $\{e_v^j\}_{v \in T_j^1}$. By Proposition~\ref{3dim4tup}, the number of ways to embed $\cup_{j \in S} \{e_v^j\}_{v \in T_j^1}$ and $V_i$ into $[m]^3$ such that $e_v^j$ crosses edge $v\hat{v}_j$ for every $j \in S$ and $v \in T_j^1$ is at most $(C_1 m^6 \log m)^{N_1}$. The number of ways the sets $T_j^1$ can be chosen is bounded above by $D^{N_1}$. The number of ways $\{\hat{v}_j\}_{j \in S}$ can be chosen in Algorithm~\ref{boundedDegAlg} is bounded above by $t^D$. The number of ways all tuples $\cup_{j \in S} \{e_v^j\}_{v \in T_j^1}$ can be chosen is at most $k^{N_1}$, since these
edges are necessarily spanned by the vertices $\{v_j\}_{j \in S}$. The probability that Step 4 cannot be completed for iteration $i$ in Case(i) is therefore at most \\
\begin{equation}\label{Casei} \frac{D^{N_1} t^D (C_1 m^6 \log m)^{N_1}k^{N_1}}{m^3(m^3 - 1) \cdots (m^3 -3N_1 + 1 )} \leq \frac{t^D(C_1D)^{N_1}(\log^{N_1}m)k^{N_1}m^{-3N_1}}{(1-3N_1m^{-3})^{3N_1}} \leq 10^{-N_1}. \end{equation}
In the second inequality above we used that $(1- 3N_1m^{-3})^{3N_1} \geq 1 - 9N_1^2 m^{-3} \geq 1/2$, and $t^{D/N_1} \leq e^{2D\log t/t } \leq 4$.\\



\emph{Case(ii):} $N_2 = \sum_{j \in S} |T_j^2| \geq t/2$. For every $j \in S$, fix arbitrary tuples of vertices in $G(t)$, indexed by the vertices of $T_j^2$: $\{u_v^j\}_{v \in T_j^2}$. By Lemma~\ref{collinearCount}, the number of ways to embed $\cup_{j \in S} \{u_v^j\}_{v \in T_j^2}$ and $V_i$ into $[m]^3$ such that $u_v^j$ is inside edge $v\hat{v}_j$ for every $j \in S$ and $v \in T_j^2$ is at most $(C_2 m^3)^{N_2}$. The number of ways all tuples $\cup_{j \in S} \{u_v^j\}_{v \in T_j^2}$ can be chosen is at most $n^{N_2}$, since these vertices belong to $\{v_j\}_{j \in S}$. The probability that Step 4 cannot be completed for iteration $i$ in Case(ii) is therefore at most \\
\begin{equation}\label{Caseii} \frac{D^{N_2} t^D (C_2 m^3)^{N_2}n^{N_2}}{m^3(m^3 - 1) \cdots (m^3 -2N_1 + 1 )} \leq \frac{t^D(C_2D)^{N_2} n^{N_2}m^{-3N_2}}{(1- 2N_2m^{-3})^{2N_2}} \leq 10^{-N_2}.\end{equation}\\

\noindent The probability that Step 4 cannot be completed for some iteration is bounded above by the sum of the quantities in \eqref{Casei}, \eqref{Caseii} multiplied by $n$. Therefore with probability at least $1 - n(10^{-N_1}+10^{-N_2}) \geq 1/2$ Algorithm~\ref{boundedDegAlg} successfully outputs a grid-drawing of $G$.


\end{proof}

\section{Open problems}\label{concSec}

\begin{enumerate}[1.]\setlength{\itemsep}{5pt}
    \item May the dependence on $D$ in Theorem~\ref{blowupThm} be improved? Can the logarithmic term be removed? A positive answer to this second question would fully resolve the problem of Pach, Thiele, and T\'oth~\cite{PTT}. Dujmovi\'c, Morin, and Sheffer~\cite{DMS} estimate $\text{ncs}_d(N)$, i.e., the number graphs that can be drawn in a grid with volume $N$ in dimension $d \geq 4$. They pose the problem of estimating $\text{ncs}_3(N)$, attributed to D. R. Wood, and show that if $\text{ncs}_3(N) = 2^{o(N \log N)}$ then not every $n$-vertex graph with maximum degree 3 can be drawn in an $O(n)$ volume grid. 

    \item Another open problem of Pach, Thiele, and T\'oth~\cite{PTT} is the following. What is the smallest $m = m(n)$ such that $K_{n,n}$ can be drawn in $[m]^3$? It is known that $K_{n,n}$ can be drawn in a $O(\sqrt{n}) \times O(\sqrt{n}) \times O(n)$ grid, and that a volume of $\Omega(n^2)$ is necessary. No grid-drawings with an aspect ratio of less than $\sqrt{n}$ are known. 

    \item In~\cite{DWood04} the authors ask if every $n$-vertex graph with $k$ edges can be drawn in a grid with volume $O(kn)$. Theorem~\ref{blowupThm} answers this affirmatively for $D$-degenerate graphs with $D = O(n \log^{-1}n)$.

    



\end{enumerate}

\section*{Acknowledgements}

The authors thank Fabrizio Frati for introducing us to various problems on three-dimen\-sional graph drawing, and thank Martin Balko and Igor Araujo for helpful conversations. The authors are also grateful to Andrew Suk for organizing \emph{Workshop on Graph Drawing and Intersection Graphs} at UCSD from January 20--21, 2024.




\end{document}